\documentclass[oneside,english]{amsart}
\usepackage[T1]{fontenc}
\usepackage[latin9]{inputenc}
\usepackage{color}
\usepackage{textcomp}
\usepackage{amsthm}
\usepackage{amssymb}
\usepackage[all]{xy}
\PassOptionsToPackage{normalem}{ulem}
\usepackage{ulem}

\makeatletter
\numberwithin{equation}{section}
\numberwithin{figure}{section}
\theoremstyle{plain}
\newtheorem{thm}{\protect\theoremname}[section]
  \theoremstyle{plain}
  \newtheorem{fact}[thm]{\protect\factname}
  \theoremstyle{definition}
  \newtheorem{example}[thm]{\protect\examplename}
  \theoremstyle{definition}
  \newtheorem{defn}[thm]{\protect\definitionname}
  \theoremstyle{plain}
  \newtheorem{lem}[thm]{\protect\lemmaname}
  \theoremstyle{definition}
  \newtheorem{problem}[thm]{\protect\problemname}
  \theoremstyle{plain}
  \newtheorem{cor}[thm]{\protect\corollaryname}
  \theoremstyle{plain}
  \newtheorem{prop}[thm]{\protect\propositionname}
  \theoremstyle{remark}
  \newtheorem{rem}[thm]{\protect\remarkname}
  \theoremstyle{plain}
  \newtheorem{conjecture}[thm]{\protect\conjecturename}
\newenvironment{lyxlist}[1]
{\begin{list}{}
{\settowidth{\labelwidth}{#1}
 \setlength{\leftmargin}{\labelwidth}
 \addtolength{\leftmargin}{\labelsep}
 }}
{\end{list}}

\usepackage[all]{xy}

\makeatother

\usepackage{babel}
  \providecommand{\conjecturename}{Conjecture}
  \providecommand{\corollaryname}{Corollary}
  \providecommand{\definitionname}{Definition}
  \providecommand{\examplename}{Example}
  \providecommand{\factname}{Fact}
  \providecommand{\lemmaname}{Lemma}
  \providecommand{\problemname}{Problem}
  \providecommand{\propositionname}{Proposition}
  \providecommand{\remarkname}{Remark}
\providecommand{\theoremname}{Theorem}

\begin{document}
\def\Ind#1#2{#1\setbox0=\hbox{$#1x$}\kern\wd0\hbox to 0pt{\hss$#1\mid$\hss}
\lower.9\ht0\hbox to 0pt{\hss$#1\smile$\hss}\kern\wd0}
\def\Notind#1#2{#1\setbox0=\hbox{$#1x$}\kern\wd0\hbox to 0pt{\mathchardef
\nn="3236\hss$#1\nn$\kern1.4\wd0\hss}\hbox to 0pt{\hss$#1\mid$\hss}\lower.9\ht0
\hbox to 0pt{\hss$#1\smile$\hss}\kern\wd0}
\def\indi{\mathop{\mathpalette\Ind{}}}
\def\nindi{\mathop{\mathpalette\Notind{}}}
\def\bdd {bdd}

\global\long\def\acl{\operatorname{acl}}

\global\long\def\id{\operatorname{id}}

\global\long\def\dcl{\operatorname{dcl}}

\global\long\def\Avg{\operatorname{Avg}}

\global\long\def\Gal{\operatorname{Gal}}

\global\long\def\inp{\operatorname{inp}}

\global\long\def\sep{\operatorname{sep}}

\global\long\def\ind{\operatorname{\indi}}

\global\long\def\indisc{\operatorname{EM}}

\global\long\def\nind{\operatorname{\nindi}}

\global\long\def\ist{\operatorname{ist}}

\global\long\def\Aut{\operatorname{Aut}}

\global\long\def\M{\operatorname{\mathbb{M}}}

\global\long\def\NTP{\operatorname{NTP}}

\global\long\def\NIP{\operatorname{NIP}}

\global\long\def\TP{\operatorname{TP}}

\global\long\def\tp{\operatorname{tp}}

\global\long\def\transp{\operatorname{T}}

\global\long\def\lstp{\operatorname{L}}

\global\long\def\NSOP{\operatorname{NSOP}}

\global\long\def\PRC{\operatorname{PRC}}

\global\long\def\ACFA{\operatorname{ACFA}}

\global\long\def\PAC{\operatorname{PAC}}

\global\long\def\Sone{\operatorname{S1}}

\global\long\def\pPC{\operatorname{PpC}}

\global\long\def\bdn{\operatorname{bdn}}

\global\long\def\alg{\operatorname{alg}}

\global\long\def\cl{\operatorname{cl}}

\global\long\def\dprk{\operatorname{dp-rk}}

\global\long\def\fund{\operatorname{fund}}

\global\long\def\div{\operatorname{div}}

\global\long\def\card{\operatorname{\mbox{Card}^{*}}}

\global\long\def\Kk{\mathcal{K}}
\global\long\def\AS{\varrho}
\global\long\def\Ff{\mathbb{F}}
\global\long\def\Zz{\mathbb{Z}}

\global\long\def\C{\mathfrak{C}}
\global\long\def\Qq{\mathbb{Q}}

\title{Groups and fields with $\mbox{NTP}_{2}$ }

\author{Artem Chernikov, Itay Kaplan and Pierre Simon}
\begin{abstract}
$\NTP_{2}$ is a large class of first-order theories defined by Shelah
and generalizing simple and NIP theories. Algebraic examples of $\NTP_{2}$
structures are given by ultra-products of $p$-adics and certain valued
difference fields (such as a non-standard Frobenius automorphism living
on an algebraically closed valued field of characteristic 0). In this
note we present some results on groups and fields definable in $\NTP_{2}$
structures. Most importantly, we isolate a chain condition for definable
normal subgroups and use it to show that any $\NTP_{2}$ field has
only finitely many Artin-Schreier extensions. We also discuss a stronger
chain condition coming from imposing bounds on burden of the theory
(an appropriate analogue of weight), and show that every strongly
dependent valued field is Kaplansky.
\end{abstract}

\thanks{The first author was partially supported by the {[}European Community's{]}
Seventh Framework Programme {[}FP7/2007-2013{]} under grant agreement
n\textdegree{} 238381.}

\thanks{The second author was supported by SFB 878.}

\maketitle

\section{Introduction}

The class of $\NTP_{2}$ theories (i.e. theories without the tree
property of the second kind) was introduced by Shelah \cite{MR595012,MR1083551}.
It generalizes both simple and NIP theories and turns out to be a
good context for the study of forking and dividing, even if one is
only interested in NIP theories: in \cite{CheKap,Chernikov:2012uq,CheBY}
it is demonstrated that the theory of forking in simple theories \cite{MR2694252}
can be viewed as a special case of the theory of forking in $\NTP_{2}$
theories over an extension base.

What are the known algebraic examples of $\NTP_{2}$ theories?
\begin{fact}
\cite{Chernikov:2012uq}\label{fac: preserving NTP2}Let $\bar{K}=\left(K,\Gamma,k,v,\mbox{ac}\right)$
be a Henselian valued field of equicharacteristic $0$ in the Denef-Pas
language. Assume that $k$ is $\NTP_{2}$ (respectively, $\Gamma$
and $k$ are strong, of finite burden --- see Section \ref{sec: Strong theories and burden}).
Then $\bar{K}$ is $\NTP_{2}$ (respectively strong, of finite burden).\end{fact}
\begin{example}
Let $\mathfrak{U}$ be a non-principal ultra-filter on the set of
prime numbers $P$. Then:
\begin{enumerate}
\item $\bar{K}=\prod_{p\in P}\mathbb{Q}_{p}/\mathfrak{U}$ is $\NTP_{2}$.
This follows from Fact \ref{fac: preserving NTP2} because:

\begin{itemize}
\item The residue field is pseudo-finite, so of burden $1$ (as burden is
bounded by weight in a simple theory by \cite{HansBurden}).
\item The value group is a $\mathbb{Z}$-group, thus dp-minimal, and burden
equals dp-rank in $\NIP$ theories by \cite{HansBurden}.
\end{itemize}

We remark that, while $\mathbb{Q}_{p}$ is dp-minimal for each $p$
by \cite{DolichGoodrickLippel}, the field $\bar{K}$ is neither simple
nor NIP even in the pure ring language (as the valuation ring is definable
by \cite{AxDefiningValuation}).

\item $\bar{K}=\prod_{p\in P}F_{p}\left(\left(t\right)\right)/\mathfrak{U}$
is $\NTP_{2}$, of finite burden, as it has the same theory as the
previous example by \cite{ax1965diophantine} (while each of $F_{p}\left(\left(t\right)\right)$
has $\TP_{2}$ by Corollary \ref{cor: F_p((t)) has TP2}).
\end{enumerate}
\end{example}
\begin{fact}
\cite{CheHils} Let $\bar{K}=\left(K,\Gamma,k,v,\mbox{ac},\sigma\right)$
be a $\sigma$-Henselian contractive valued difference field of equicharacteristic
$0$, i.e. $\sigma$ is an automorphism of the field $K$ such that
for all $x\in K$ with $v\left(x\right)>0$ we have $v\left(\sigma\left(x\right)\right)>n\cdot v\left(x\right)$
for all $n\in\omega$ (see \cite{MR2725200}). Assume that both $\left(K,\sigma\right)$
and $\left(\Gamma,\sigma\right)$, with the naturally induced automorphisms,
are $\NTP_{2}$. Then $\bar{K}$ is $\NTP_{2}$.\end{fact}
\begin{example}
Let $\left(F_{p},\Gamma,k,v,\sigma\right)$ be an algebraically closed
valued field of characteristic $p$ with $\sigma$ interpreted as
the Frobenius automorphism. Then $\prod_{p\in P}F_{p}/\mathfrak{U}$
is $\NTP_{2}$. This case was studied by Hrushovski \cite{Hrushovski:nr}
and later by Durhan \cite{MR2725200}. It follows from\textcolor{red}{{}
}\cite{Hrushovski:nr} that the reduct to the field language is a
model of $\ACFA$, hence simple but not $\NIP$. On the other hand
this theory is not simple as the valuation group is definable.
\end{example}
Moreover, certain valued difference fields with a value preserving
automorphism are $\NTP_{2}$. Of course, any simple or $\NIP$ field
is $\NTP_{2}$, and there are further conjectural examples of pure
$\NTP_{2}$ fields such as bounded pseudo real closed or pseudo $p$-adically
closed fields (see Section \ref{sub: more NTP2 fields}).

~

But what does knowing that a theory is $\NTP_{2}$ tell us about properties
of algebraic structures definable in it? In this note we show some
initial implications. In Section \ref{sec: Chain conditions} we isolate
a chain condition for normal subgroups uniformly definable in an $\NTP_{2}$
theory. In Section \ref{sec: Fields with NTP2} we use it to demonstrate
that every field definable in an $\NTP_{2}$ theory has only finitely
many Artin-Schreier extensions, generalizing some of the results of
\cite{KaScWa}. In Section \ref{sec: Strong theories and burden}
we impose bounds on the burden, a quantitative refinement of $\NTP_{2}$
similar to SU-rank in simple theories, and observe that some results
for type-definable groups existing in the literature actually go through
with a weaker assumption of bounded burden, e.g. every strong field
is perfect, and every strongly dependent valued field is Kaplansky.
The final section contains a discussion around the topics of the paper:
we pose several conjectures about new possible examples (and non-examples)
of $\NTP_{2}$ fields and about definable envelopes of nilpotent/soluble
groups in $\NTP_{2}$ theories. We also remark how the stabilizer
theorem of Hrushovski from \cite{MR2833482} could be combined with
properties of forking established in \cite{CheKap} and \cite{CheBY}
in the $\NTP_{2}$ context.

~

We would like to thank Arno Fehm for his comments on Section \ref{sub: more NTP2 fields}
and Example \ref{ex:Arno}. We would also like to that the anonymous
referee for many useful remarks.

\subsection*{Preliminaries}

Our notation is standard. As usual, we will be working in a monster
model $\C$ of the theory under consideration. Let $G$ be a group,
and $H$ a subgroup of $G$. We write $\left[G:H\right]<\infty$ to
denote that the index of $H$ in $G$ is bounded, which in the case
of definable groups means finite. We assume that all groups (and fields)
are finitary --- contained in some finite Cartesian product of the
monster. 
\begin{defn}
We recall that a formula $\varphi\left(x,y\right)$ has $\TP_{2}$
if there are tuples $\left(a_{i,j}\right)_{i,j\in\omega}$ and $k\in\omega$
such that:
\begin{itemize}
\item $\left\{ \varphi\left(x,a_{i,j}\right)\left|\, j<\omega\right.\right\} $
is $k$-inconsistent, for each $i\in\omega$,
\item $\left\{ \varphi\left(x,a_{i,f\left(i\right)}\right)\left|\, i<\omega\right.\right\} $
is consistent for each $f:\,\omega\to\omega$.
\end{itemize}
\end{defn}
A formula is $\NTP_{2}$ otherwise, and a theory is called $\NTP_{2}$
if no formula has $\TP_{2}$.
\begin{fact}
\cite{Chernikov:2012uq} $T$ is $\NTP_{2}$ if and only if every
formula $\varphi\left(x,y\right)$ with $\left|x\right|=1$ is $\NTP_{2}$.
\end{fact}
We note that every simple or $\NIP$ formula is $\NTP_{2}$ . See
\cite{Chernikov:2012uq} for more on $\NTP_{2}$ theories.

\section{\label{sec: Chain conditions}Chain conditions for groups with $\NTP_{2}$}
\begin{lem}
\label{prop: Intersection of normal subgroups in NTP2} Let $T$ be
$\NTP_{2}$, $G$ a definable group and $\left(H_{i}\right)_{i\in\omega}$
a uniformly definable family of normal subgroups of $G$, with $H_{i}=\varphi\left(x,a_{i}\right)$.
Let $H=\bigcap_{i\in\omega}H_{i}$ and $H_{\neq j}=\bigcap_{i\in\omega\setminus\left\{ j\right\} }H_{i}$.
Then there is some $i^{*}\in\omega$ such that $\left[H_{\neq i^{*}}:H\right]$
is finite.\end{lem}
\begin{proof}
Let $\left(H_{i}\right)_{i\in\omega}$ be given and assume that the
conclusion fails. Then for each $i\in\omega$ we can find $\left(b_{i,j}\right)_{j\in\omega}$
with $b_{i,j}\in H_{\neq i}$ and such that $\left(b_{i,j}H\right)_{j\in\omega}$
are pairwise different cosets in $H_{\neq i}$. We have:
\begin{itemize}
\item $b{}_{i,j}H_{i}\cap b{}_{i,k}H_{i}=\emptyset$ for $j\neq k\in\omega$
and every $i$. 
\item For every $f:\omega\to\omega$, the intersection $\bigcap_{i\in\omega}b_{i,f\left(i\right)}H_{i}$
is non-empty. Indeed, fix $f$, by compactness it is enough to check
that $\bigcap_{i\leq n}b_{i,f\left(i\right)}H_{i}\neq\emptyset$ for
every $n\in\omega$. Take $b=\prod_{i\leq n}b_{i,f\left(i\right)}$
(the order of the product does not matter). As $b_{i,f\left(i\right)}\in H_{j}$
for all $i\neq j$, it follows by normality that $b\in b_{i,f\left(i\right)}H_{i}$
for all $i\leq n$. 
\end{itemize}
But then $\psi\left(x;y,z\right)=\exists w\left(\varphi\left(w,y\right)\land x=z\cdot w\right)$
has $\TP_{2}$ as witnessed by the array $\left(c_{i,j}\right)_{i,j\in\omega}$
with $c_{i,j}=a_{i}\hat{\,}b_{i,j}.$\end{proof}
\begin{problem}
Is the same result true without the normality assumption? See also
Theorem \ref{thm: without normality in strong^2}.\end{problem}
\begin{cor}
\label{cor: finite intersection} Let $T$ be $\NTP_{2}$ and suppose
that $G$ is a definable group. Then for every $\varphi\left(x,y\right)$
there are $k_{\varphi},n_{\varphi}\in\omega$ such that:\end{cor}
\begin{itemize}
\item If $\left(\varphi\left(x,a_{i}\right)\right)_{i<K}$ is a family of
normal subgroups of\textcolor{red}{{} }$G$ and $k_{\varphi}\leq K$
then there is some $i^{*}<K$ such that $\left[\bigcap_{i<K,i\neq i^{*}}\varphi\left(x,a_{i}\right):\bigcap_{i<K}\varphi\left(x,a_{i}\right)\right]<n_{\varphi}$.\textcolor{red}{{} }\end{itemize}
\begin{proof}
Follows from Lemma \ref{prop: Intersection of normal subgroups in NTP2}
and compactness. \end{proof}
\begin{thm}
\label{Thm:Weak-Baldwin-Saxl} Let $G$ be $\NTP_{2}$ and $\left\{ \varphi\left(x,a\right)\left|\, a\in C\right.\right\} $
be a family of normal subgroups of $G$. Then there is some $k\in\omega$
(depending only on $\varphi$) such that for every finite $C'\subseteq C$
there is some $C_{0}\subseteq C'$ with $\left|C_{0}\right|\leq k$
and such that 
\[
\left[\bigcap_{a\in C_{0}}\varphi\left(x,a\right):\bigcap_{a\in C'}\varphi\left(x,a\right)\right]<\infty\mbox{.}
\]
\end{thm}
\begin{proof}
Let $k_{\varphi}$ be as given by Corollary \ref{cor: finite intersection}.
If $\left|C'\right|>k_{\varphi}$, by Corollary \ref{cor: finite intersection}
we find some $a_{0}\in C'$ such that $\left[\bigcap_{a\in C'\setminus\left\{ a_{0}\right\} }\varphi\left(x,a\right)\,:\,\bigcap_{a\in C'}\varphi\left(x,a\right)\right]<\infty$.
If $\left|C'\setminus\left\{ a_{0}\right\} \right|>k_{\varphi}$,
by Corollary \ref{cor: finite intersection} again we find some $a_{1}\in C'\setminus\left\{ a_{0}\right\} $
such that 
\[
\left[\bigcap_{a\in C'\setminus\left\{ a_{0},a_{1}\right\} }\varphi\left(x,a\right):\bigcap_{a\in C'\setminus\left\{ a_{0}\right\} }\varphi\left(x,a\right)\right]<\infty\mbox{.}
\]
Continuing in this way we end up with $a_{0},\ldots,a_{m}\in C'$
such that\textcolor{red}{{} }for all\textcolor{red}{{} }$i<m$, 
\[
\left[\bigcap_{a\in C'\setminus\left\{ a_{0},\ldots,a_{i+1}\right\} }\varphi\left(x,a\right):\bigcap_{a\in C'\setminus\left\{ a_{0},\ldots,a_{i}\right\} }\varphi\left(x,a\right)\right]<\infty,
\]
and, letting $C_{0}=C'\setminus\left\{ a_{0},\ldots,a_{m}\right\} $,
we have that $\left|C_{0}\right|\leq k_{\varphi}$.\end{proof}
\begin{cor}
Let $G$ be a torsion-free group with $\mbox{NTP}_{2}$ and assume
that $\varphi(x,y)$ defines a divisible normal subgroup for every
$y$. Then $\varphi(x,y)$ is $\mbox{NIP}$.\end{cor}
\begin{proof}
Assume that $\varphi(x,y)$ has IP and let $\bar{a}=(a_{i})_{i\in\mathbb{Z}}$
be an indiscernible sequence witnessing this. Taking $H_{i}=\varphi(\C,a_{i})$,
$H_{\neq0}\setminus H_{0}\neq\emptyset$ . Let $H=\bigcap_{i\in\mathbb{Z}}H_{i}$,
so it is divisible (here we used the assumption that $G$ is torsion-free)
as is $H_{\neq0}$. But then $H_{\neq0}/H$ is a divisible non-trivial
group, so infinite. By indiscernibility $\left[H_{\neq i}:H\right]=\infty$
for all $i$ --- contradicting Lemma \ref{prop: Intersection of normal subgroups in NTP2}.
\end{proof}

\section{\label{sec: Fields with NTP2}Fields with $\NTP_{2}$}

Let $K$ be a field of characteristic $p>0$. Recall that a field
extension $L/K$ is called an Artin-Schreier extension if $L=K\left(\alpha\right)$
for some $\alpha\in L\setminus K$ such that $\alpha^{p}-\alpha\in K$.
$L/K$ is an Artin-Schreier extension if and only if it is Galois
and cyclic of degree $p$.
\begin{thm}
\label{thm: finitely many A-S extensions in NTP2}Let $K$ be an infinite
field definable in an $\NTP_{2}$ theory. Then it has only finitely
many Artin-Schreier extensions.\end{thm}
\begin{proof}
We follow the proof of the fact that dependent fields have no Artin-Schreier
extensions in \cite{KaScWa}. 

We may assume that $K$ is $\aleph_{0}$-saturated, and we put $k=K^{p^{\infty}}=\bigcap_{n\in\omega}K^{p^{n}}$,
a type-definable perfect sub-field which is infinite by saturation
(all contained in an algebraically closed $\Kk$).

For a tuple $\bar{a}=\left(a_{0},\ldots,a_{n-1}\right)$, let 
\[
G_{\bar{a}}=\left\{ \left(t,x_{0},\ldots,x_{n-1}\right)\in K^{n+1}:\, t=a_{i}\cdot\AS\left(x_{i}\right)\mbox{ for }i<n\right\} \mbox{,}
\]
where $\AS\left(x\right)=x^{p}-x$ is the Artin-Schreier polynomial.
We consider it as an algebraic group (a subgroup of $\left(\Kk^{n+1},+\right)$).
As such, by \cite[Lemma 2.8]{KaScWa}, when the elements of $\bar{a}$
are algebraically independent it is connected. If in addition $\bar{a}$
belong to some perfect field $k$, then $G_{\bar{a}}$ is isomorphic
by an algebraic isomorphism over $k$ to $\left(\Kk,+\right)$ by
\cite[Corollary 2.9]{KaScWa}. 

By Theorem \ref{Thm:Weak-Baldwin-Saxl}, there is some $n<\omega$,
an algebraically independent $\left(n+1\right)$-tuple $\bar{a}\in k$
and an $n$-subtuple $\bar{a}'$, such that $\left[\bigcap_{a\in\bar{a}'}a\cdot\AS\left(K\right):\bigcap_{a\in\bar{a}}a\cdot\AS\left(K\right)\right]<\infty$.
It follows that the image of the projection map $\pi:G_{\bar{a}}\left(K\right)\to G_{\bar{a}'}\left(K\right)$
has finite index in $G_{\bar{a}'}\left(K\right)$. 

We have algebraic isomorphisms $G_{\bar{a}}\rightarrow(\Kk,+)$ and
$G_{\bar{a}'}\rightarrow(\Kk,+)$ over $k$. Hence we can find an
algebraic map $\rho$ over $k$ (i.e. a polynomial) which makes the
following diagram commute: 
\[
\xymatrix{G_{\bar{a}}\ar[r]^{\pi}\ar[d] & G_{\bar{a}'}\ar[d]\\
\left(\Kk,+\right)\ar[r]^{\rho} & \left(\Kk,+\right)
}
\]

As all groups and maps are defined over $k\subseteq K$, we can restrict
to $K$. We saw that $\left[G_{\bar{a}'}:\pi\left(G_{\bar{a}}\left(K\right)\right)\right]<\infty$
so $\left[K:\rho\left(K\right)\right]<\infty$ as well (in the group
$\left(K,+\right)$). In the proof of \cite[Theorem 4.3]{KaScWa},
it is shown that there is some $c\in K$ such that, letting $\rho'\left(x\right)=\rho\left(c\cdot x\right)$,
$\rho'$ has the form $a\cdot\AS\left(x\right)$ for some $a\in K^{\times}$.
The way it is done there is by choosing any $0\neq c\in\ker\left(\rho\right)\subseteq k$,
and then since $\rho'$ is additive with kernel $\Ff_{p}$ and degree
$p$ (as this is the degree of $\pi$), there exists such an $a\in k$.
Since $\rho'\left(K\right)=\rho\left(K\right)$ has finite index in
$K$, so does the image of $\AS=a^{-1}\rho'$.

By \cite[Remark 2.3]{KaScWa}, this index is finite if and only if
the number of Artin-Schreier extensions is finite.\end{proof}
\begin{prop}
\label{prop: finitely many AS extensions implies p-divisible}Suppose
$\left(K,v,\Gamma\right)$ is a valued field of characteristic $p>0$
that has finitely many Artin-Schreier extensions. Then the valuation
group $\Gamma$ is $p$-divisible.\end{prop}
\begin{proof}
(This proof is similar to the proof of \cite[Proposition 5.4]{KaScWa}.)
Recall that $\AS$ is the Artin-Schreier polynomial. By Artin-Schreier
theory (this is explained in \cite[Remark 2.3]{KaScWa}), the index
$\left[K:\AS\left(K\right)\right]$ in the additive group $\left(K,+\right)$
is finite. Suppose $\left\{ a_{i}\left|\, i<l\right.\right\} $ are
representatives for the cosets of $\AS\left(K\right)$ in $\left(K,+\right)$.
Let $\alpha\in\Gamma$ be smaller than $\alpha_{0}=\min\left\{ v\left(a_{i}\right)\left|\, i<l\right.\right\} \cup\left\{ 0\right\} $.
Suppose $v\left(x\right)=\alpha$ for $x\in K$. But then there is
some $i<l$ such that $x-a_{i}\in\AS\left(K\right)$, and since $v\left(x\right)=v\left(x-a_{i}\right)$,
we may assume that $x\in\AS\left(K\right)$, so there is some $y$
such that $y^{p}-y=x$. But then, $v\left(y\right)<0$, so $v\left(y^{p}\right)=p\cdot v\left(y\right)<v\left(y\right)$,
so 
\[
\alpha=v\left(x\right)=v\left(y^{p}-y\right)=v\left(y^{p}\right)=p\cdot v\left(y\right).
\]
So $\alpha$ is $p$-divisible. Take any negative $\beta\in\Gamma$,
then $\beta+p\cdot\alpha_{0}$ is $p$-divisible, so $\beta$ is also
$p$-divisible. Since this is true for all negative values, $\Gamma$
is $p$-divisible.\end{proof}
\begin{cor}
\label{cor: F_p((t)) has TP2}$\Ff_{p}\left(\left(t\right)\right)$
is not $\NTP_{2}$.\end{cor}
\begin{proof}
Follows from Theorem \ref{thm: finitely many A-S extensions in NTP2}
and Proposition \ref{prop: finitely many AS extensions implies p-divisible}.
\end{proof}

\section{\label{sec: Strong theories and burden}Strong theories and bounded
burden}

In this section we are going to consider groups and fields whose theories
satisfy quantitative refinements of $\NTP_{2}$ in terms of a bound
on its burden (similar to the bounds on the rank in simple theories).

For notational convenience we consider an extension $\card$ of the
linear order on cardinals by adding a new maximal element $\infty$
and replacing every limit cardinal $\kappa$ by two new elements $\kappa_{-}$
and $\kappa_{+}$. The standard embedding of cardinals into $\card$
identifies $\kappa$ with $\kappa_{+}$. In the following, whenever
we take a supremum of a set of cardinals, we will be computing it
in $\card$.
\begin{defn}
Let $T$ be a complete theory.
\begin{enumerate}
\item An $\inp$-pattern of depth $\kappa$ consists of $\left(\bar{a}_{\alpha},\varphi_{\alpha}(x,y_{\alpha}),k_{\alpha}\right)_{\alpha\in\kappa}$
with $\bar{a}_{\alpha}=\left(a_{\alpha,i}\right)_{i\in\omega}$ and
$k_{\alpha}\in\omega$ such that:

\begin{itemize}
\item $\left\{ \varphi_{\alpha}\left(x,a_{\alpha,i}\right)\left|\, i<\omega\right.\right\} $
is $k_{\alpha}$-inconsistent for every $\alpha\in\kappa$,
\item $\left\{ \varphi_{\alpha}\left(x,a_{\alpha,f(\alpha)}\right)\left|\,\alpha<\kappa\right.\right\} $
is consistent for every $f:\,\kappa\to\omega$.
\end{itemize}
\item An inp$^{2}$-pattern of depth $\kappa$ consists of $\left(\bar{a}_{\alpha},b_{\alpha},\phi_{\alpha}\left(x,y_{\alpha},z_{\alpha}\right)\right)_{\alpha<\kappa}$,
where $\phi_{\alpha}\in L$, $\bar{a}_{\alpha}=\left(a_{\alpha,i}\right)_{i<\omega}$,
and $b_{\alpha}\subseteq\bigcup\left\{ \bar{a}_{\beta}\left|\,\beta<\alpha\right.\right\} $,
such that:

\begin{itemize}
\item $\left(\bar{a}_{\alpha}\right)_{\alpha<\kappa}$ are mutually indiscernible.
\item $\left\{ \phi_{\alpha}\left(x,a_{\alpha,i},b_{\alpha}\right)\left|\, i<\omega\right.\right\} $
is inconsistent for every $\alpha$,
\item $\left\{ \phi_{\alpha}\left(x,a_{\alpha,0},b_{\alpha}\right)\left|\,\alpha<\kappa\right.\right\} $
is consistent.
\end{itemize}
\item The \emph{burden (burden$^{2}$)} of $T$ is the supremum (in Card$^{*}$)
of the depths of $\inp$-patterns (resp. $\inp^{2}$-patterns) with
$x$ a singleton.
\item It is easy to see by compactness that $T$ is $\NTP_{2}$ if and only
if its burden is $<\infty$, equivalently $<\left|T\right|^{+}$.
The same is true for burden$^{2}$, see \cite[Proposition 5.5(viii)]{CheBY}.
\item A theory $T$ is called \emph{strong (strong$^{2}$)} if its burden
$\leq\left(\aleph_{0}\right)_{-}$ (resp. burden$^{2}$$\leq\left(\aleph_{0}\right)_{-}$).
\end{enumerate}
\end{defn}
Strong theories were introduced by Adler \cite{HansBurden} based
on the notion of $\inp$-patterns of Shelah \cite[Ch. III]{MR1083551},
and were further studied in \cite{Chernikov:2012uq} where it was
shown that burden is ``sub-multiplicative''. Strong$^{2}$ theories
were introduced in \cite{CheBY} as a generalization of Shelah's strongly$^{2}$
dependent theories. Of course, every strong$^{2}$ theory is strong,
and every strong theory is $\NTP_{2}$.
\begin{fact}
\label{fac: Burden-is-sub-multiplicative}\cite{Chernikov:2012uq}Burden
is ``sub-multiplicative'': if there is an $\inp$-pattern of depth
$\kappa^{n}$ with $\left|x\right|=n$ then there is an $\inp$-pattern
of depth $\kappa$ with $\left|x\right|=1$. In particular, in a strong
theory there are no $\inp$-patterns of infinite depth with $x$ of
arbitrary finite length (while the definition only requires this for
$\left|x\right|=1$).\end{fact}
\begin{problem}
Does the same hold for $\inp^{2}$-patterns?\end{problem}
\begin{rem}

\begin{enumerate}
\item For $T$ simple, being strong corresponds to the fact that every finitary
type has finite weight \cite{HansBurden}. Also, every supersimple
theory is strong$^{2}$ \cite[Section 5]{CheBY}.
\item In \cite{ShelahDependentCont}, Shelah introduced strongly and strongly$^{2}$
dependent theories. For strong dependence, the definition is very
similar to the one given: one asks that there is no pattern $\left(\bar{a}_{\alpha},\varphi_{\alpha}(x,y_{\alpha})\right)_{\alpha<\omega}$
as above such that for every function $f:\omega\to\omega$, $\left\{ \varphi_{\alpha}\left(x,a_{\alpha,\beta}\right)^{\mbox{if }\beta=f\left(\alpha\right)}\left|\,\alpha<\kappa\right.\right\} $
is consistent. One can easily show that $T$ is strongly dependent
if and only if it is strong and $\NIP$.
\item The definition of strongly$^{2}$ dependent is again similar to the
definition of strong$^{2}$, allowing parameters from other rows in
the definition of strong dependence. For $T$ dependent, being strong$^{2}$
is the same as being strongly$^{2}$ dependent (sometimes called strongly$^{+}$
dependent)\cite[Section 5]{CheBY}.
\item There are stable strong theories which are not strong$^{2}$ and there
are stable strong$^{2}$ theories which are not superstable \cite[Section 5]{CheBY}.
\end{enumerate}
\end{rem}

\subsection{Strong groups and fields}

The following are taken from \cite[Proposition 3.11, Corollary 3.12]{Kaplan:2011xy}
with some easy modifications:
\begin{prop}
\label{prop:NormalIntersection} Let $G$ be a type-definable group
and $G_{i}\leq G$ type-definable \uline{normal} subgroups for
$i<\omega$.
\begin{enumerate}
\item If $T$ is strong, then there is some $i_{0}$ such that $\left[\bigcap_{i\neq i_{0}}G_{i}:\bigcap_{i<\omega}G_{i}\right]<\infty$.
\item If $T$ is of finite burden, then there is some $n\in\omega$ and
$i_{0}<n$ such that $\left[\bigcap_{i\neq i_{0},i<n}G_{i}:\bigcap_{i<n}G_{i}\right]<\infty$. 
\end{enumerate}
\end{prop}
\begin{proof}
(1) Assume not. Then, for each $i<\omega$, we have an indiscernible
sequence $\left(a_{i,j}\right)_{j<\omega}$ (over the parameters defining
all the groups) such that $a_{i,j}\in\bigcap_{k\neq i}G_{k}$ and
for $j_{1}<j_{2}<\omega$, $a_{i,j_{1}}^{-1}\cdot a_{i,j_{2}}\notin G_{i}$.
By compactness there is a formula $\psi_{i}\left(x\right)$ in the
type defining $G_{i}$ such that $\neg\psi_{i}\left(a_{i,j_{1}}^{-1}\cdot a_{i,j_{2}}\right)$
holds (by indiscernibility it is the same for all $j_{1}<j_{2}$).
We may assume, applying Ramsey, that the sequences $\left\langle \left(a_{i,j}\right)_{j<\omega}\left|\, i<\omega\right.\right\rangle $
are mutually indiscernible. Let $\psi_{i}'$ be another formula in
the type defining $G_{i}$ such that $\psi_{i}'\left(x\right)\wedge\psi_{i}'\left(y\right)\vdash\psi_{i}\left(x^{-1}\cdot y\right)$.
Let $\varphi_{i}\left(x,y\right)=\psi'_{i}\left(x^{-1}\cdot y\right)$. 

Now we check that the set $\left\{ \varphi_{i}\left(x,a_{i,0}\right)\left|\, i<n\right.\right\} $
is consistent for each $n<\omega$. Let $c=a_{0,0}\cdot\ldots\cdot a_{n-1,0}$
(the order does not really matter, but for the proof it is easier
to fix one). So $\varphi_{i}\left(c,a_{i,0}\right)$ holds if and
only if $\psi_{i}'\left(a_{n-1,0}^{-1}\cdot\ldots\cdot a_{i,0}^{-1}\cdot\ldots\cdot a_{0,0}^{-1}\cdot a_{i,0}\right)$
holds. But since $G_{i}$ is normal, $a_{i,0}^{-1}\cdot\ldots\cdot a_{0,0}^{-1}\cdot a_{i,0}\in G_{i}$,
so the entire product is in $G_{i}$, so $\varphi_{i}\left(c,a_{i,0}\right)$
holds. On the other hand, if for some $c'$, $\varphi_{i}\left(c',a_{i,0}\right)\land\varphi_{i}\left(c',a_{i,1}\right)$
holds, then $\psi_{i}\left(a_{i,0}^{-1}\cdot a_{i,1}\right)$ holds
--- contradiction. So the rows are inconsistent which contradicts
strength.

(2) Follows from the proof of (1) using Fact \ref{fac: Burden-is-sub-multiplicative}.\end{proof}
\begin{cor}
\label{cor:almost all}If $G$ is an abelian group type-definable
in a strong theory and $S\subseteq\omega$ is an infinite set of pairwise
co-prime numbers, then for almost all (i.e. for all but finitely many)
$n\in S$, $\left[G:G^{n}\right]<\infty$. In particular, if $K$
is a definable field in a strong theory, then for almost all primes
$p$, $\left[K^{\times}:\left(K^{\times}\right)^{p}\right]<\infty$.\end{cor}
\begin{proof}
Let $K\subseteq S$ be the set of $n\in S$ such that $\left[G:G^{n}\right]<\infty$.
If $S\backslash K$ is infinite, we replace $S$ with $S\backslash K$.

For $i\in S$, let $G_{i}=G^{i}$ (so it is type-definable). By Proposition
\ref{prop:NormalIntersection}, there is some $n$ such that $\left[\bigcap_{i\neq n}G_{i}:\bigcap_{i\in S}G_{i}\right]<\infty$.
Now it is enough to show that $\bigcap_{i\neq n}G_{i}/\bigcap_{i\in S}G_{i}\cong G/G_{n}$.
For this we show that the natural map $\bigcap_{i\neq n}G_{i}\to G/G_{n}$
is onto. To show that, we may assume by compactness that $S$ is finite.
Let $r=\prod S\backslash\left\{ n\right\} $, then since $r$ and
$n$ are co-prime, there are some $a,b\in\mathbb{Z}$ such that $ar+bn=1$
so for any $g\in G$, $g^{ar}\equiv g\pmod{G_{n}}$, and we are done. 
\end{proof}
The proof of the following proposition is taken from \cite[Proposition 2.3]{MR2787692}
so we observe that it goes through in larger generality.
\begin{prop}
\label{prop: Any-infinite-strong-field-is-perfect}Any infinite strong
field is perfect.\end{prop}
\begin{proof}
Let $K$ be of characteristic $p>0$, and suppose that $K^{p}\neq K$.
Then there are $b_{1},b_{2}\in K$ linearly independent over $K^{p}$.
Let $\left\langle a_{i}:\, i\in\mathbb{Q}\right\rangle $ be an indiscernible
non-constant sequence over $b_{1},b_{2}$. By compactness we can find
$a$ and $\left(c_{i}\right)_{i<\omega}$ from $K$ such that $c_{0}=a$
and $c_{i}=b_{1}c_{i+1}^{p}+b_{2}a_{i}^{p}$. Since $b_{1},b_{2}$
are linearly independent over $K^{p}$, we get that $a_{i}\in\dcl\left(b_{1}b_{2}a\right)$
for every $i<\omega$. For each $i<\omega$, let $\varphi_{i}\left(y,b_{1},b_{2},a\right)$
be a formula defining $a_{i}$. We may assume that $\forall x,y_{1},y_{2}\bigwedge_{j=1,2}\varphi_{i}\left(y_{j},b_{1},b_{2},x\right)\to y_{1}=y_{2}$.
So:
\begin{itemize}
\item The sequences $I_{i}=\left(a_{j}\right)_{i-1/2<j<i+1/2}$ where $i<\omega$
are mutually indiscernible over $b_{1},b_{2}$.
\item $\left\{ \varphi_{i}\left(a_{j},b_{1},b_{2},x\right)\left|\, i-1/2<j<i+1/2\right.\right\} $
is $2$-inconsistent. 
\item $\left\{ \varphi_{i}\left(a_{i},b_{1},b_{2},x\right)\left|\, i<\omega\right.\right\} $
is consistent (realized by $a$).
\end{itemize}
Which contradicts strength. \end{proof}
\begin{defn}
A valued field $\left(K,v\right)$ of characteristic $p>0$ is \emph{Kaplansky}
if it satisfies:
\begin{enumerate}
\item The valuation group $\Gamma$ is $p$-divisible.
\item The residue field $k$ is perfect, and does not admit a finite separable
extension whose degree is divisible by $p$.
\end{enumerate}
\end{defn}
\begin{cor}
Every strongly dependent valued field is Kaplansky.\end{cor}
\begin{proof}
Combining Proposition \ref{prop: Any-infinite-strong-field-is-perfect},
Proposition \ref{prop: finitely many AS extensions implies p-divisible}
and \cite[Corollary 4.4]{KaScWa}.
\end{proof}

\subsection{Strong$^{2}$ theories}

The following is just a repetition of \cite[Proposition 2.5]{Kaplan:2011xy}:
\begin{prop}
\label{prop:noInfiniteChain}Suppose $T$ is strong$^{2}$, then it
is impossible to have a sequence of type-definable groups $\left\langle G_{i}\left|\, i<\omega\right.\right\rangle $
such that $G_{i+1}\leq G_{i}$ and $\left[G_{i}:G_{i+1}\right]=\infty$.\end{prop}
\begin{proof}
Without loss of generality, we shall assume that all groups are type-definable
over $\emptyset$. Suppose there is such a sequence $\left\langle G_{i}\left|\, i<\omega\right.\right\rangle $.
Let $\left\langle \bar{a}_{i}\left|\, i<\omega\right.\right\rangle $
be mutually indiscernible, where $\bar{a}_{i}=\left\langle a_{i,j}\left|\, j<\omega\right.\right\rangle $,
such that for $i<\omega$, the sequence $\left\langle a_{i,j}\left|\, j<\omega\right.\right\rangle $
is a sequence from $G_{i}$ (in $\C$) such that $a_{i,j'}^{-1}\cdot a_{i,j}\notin G_{i+1}$
for all $j<j'<\omega$. We can find such an array because of our assumption
and Ramsey.

For each $i<\omega$, let $\psi_{i}\left(x\right)$ be in the type
defining $G_{i+1}$ such that $\neg\psi_{i}\left(a_{i,j'}^{-1}\cdot a_{i,j}\right)$
for $j'<j$. By compactness, there is a formula $\xi_{i}\left(x\right)$
in the type defining $G_{i+1}$ such that for all $a,b\in\C$, if
$\xi_{i}\left(a\right)\land\xi_{i}\left(b\right)$ then $\psi_{i}\left(a\cdot b^{-1}\right)$
holds. Let $\varphi_{i}\left(x,y,z\right)=\xi_{i}\left(y^{-1}\cdot z^{-1}\cdot x\right)$.
For $i<\omega$, let $b_{i}=a_{0,0}\cdot\ldots\cdot a_{i-1,0}$ (so
$b_{0}=1$). 

Let us check that the set $\left\{ \varphi_{i}\left(x,a_{i,0},b_{i}\right)\left|\, i<\omega\right.\right\} $
is consistent. Let $i_{0}<\omega$, and let $c=b_{i_{0}}$. Then for
$i<i_{0}$, $\varphi_{i}\left(c,a_{i,0},b_{i}\right)$ holds if and
only if $\xi_{i}\left(a_{i+1,0}\cdot\ldots\cdot a_{i_{0}-1,0}\right)$
holds, but the product $a_{i+1,0}\cdot\ldots\cdot a_{i_{0}-1,0}$
is an element of $G_{i+1}$ and $\xi_{i}$ is in the type defining
$G_{i+1}$, so $\varphi_{i}\left(c,a_{i,0},b_{i}\right)$ holds.  Now,
if $\varphi_{i}\left(c',a_{i,0},b_{i}\right)\wedge\varphi_{i}\left(c',a_{i,0},b_{i}\right)$
holds for some $c'$, then $\xi_{i}\left(a_{i,0}^{-1}b_{i}^{-1}c'\right)$
and $\xi_{i}\left(a_{i,1}^{-1}b_{i}^{-1}c'\right)$ hold, so also
$\psi_{i}\left(a_{i,0}^{-1}a_{i,1}\right)$ holds --- a contradiction.
So the rows are inconsistent, contradicting strength$^{2}$. 
\end{proof}
We also get (exactly as \cite[Proposition 2.6]{Kaplan:2011xy}):
\begin{cor}
Assume $T$ is strong$^{2}$. If $G$ is a type-definable group and
$h$ is a definable homomorphism $h:G\to G$ with finite kernel then
$h$ is almost onto $G$, i.e., the index $\left[G:h\left(G\right)\right]$
is bounded (i.e. $<\infty$). If $G$ is definable, then the index
must be finite.
\end{cor}
Theorem \ref{Thm:Weak-Baldwin-Saxl} holds for type-definable subgroups
without the normality assumption.
\begin{thm}
\label{thm: without normality in strong^2}Let $G$ be strong$^{2}$
and $\left\{ \varphi\left(x,a\right)\left|\, a\in C\right.\right\} $
be a family of definable subgroups of $G$. Then there is some $k\in\omega$
such that for every finite $C'\subseteq C$ there is some $C_{0}\subseteq C'$
with $\left|C_{0}\right|\leq k$ and such that 
\[
\left[\bigcap_{a\in C_{0}}\varphi\left(x,a\right):\bigcap_{a\in C'}\varphi\left(x,a\right)\right]<\infty\mbox{.}
\]
\end{thm}
\begin{proof}
The proof of Theorem \ref{Thm:Weak-Baldwin-Saxl} relied on Proposition
\ref{prop: Intersection of normal subgroups in NTP2}. So we only
need to show that this proposition goes through. Let $H_{i}=\varphi\left(x,a_{i}\right)$
for $i<\omega$. Consider $H_{i}'=\bigcap_{j<i}H_{j}$. At some point
$\left[H'_{j}:H'_{j+1}\right]<\infty$. But then also $\left[H_{\neq j}:\bigcap_{i<\omega}H_{i}\right]<\infty$. 
\end{proof}

\section{Questions, conjectures and further research directions}

\subsection{\label{sub: more NTP2 fields}More pure $\NTP_{2}$ fields}

Recall that a field is pseudo algebraically closed (or $\PAC$) if
every absolutely irreducible variety defined over it has a point in
it. It is well-known \cite{MR1721163} that the theory of a $\PAC$
field is simple if and only if it is bounded (i.e. for any integer
$n$ it has only finitely many Galois extensions of degree $n$).
Moreover, if a $\PAC$ field is unbounded, then it has $\TP_{2}$
\cite[Section 3.5]{ZoeUnboundedPAC}. On the other had, the following
fields were studied extensively:
\begin{enumerate}
\item Pseudo real closed (or $\PRC$) fields: a field $F$ is $\PRC$ if
every absolutely irreducible variety defined over $F$ that has a
rational point in every real closure of $F$, has an $F$-rational
point \cite{MR1071779,MR645909,MR799048}.
\item Pseudo $p$-adically closed (or $\pPC$) fields: a field $F$ is $\pPC$
if every absolutely irreducible variety defined over $F$ that has
a rational point in every $p$-adic closure of $F$, has an $F$-rational
point \cite{MR1006730,MR1087579,MR1133080,MR921990}.\end{enumerate}
\begin{conjecture}
A $\PRC$ field is $\NTP_{2}$ if and only if it is bounded. Similarly,
a $\pPC$ field is $\NTP_{2}$ if and only if it is bounded.
\end{conjecture}
We remark that if $K$ is an unbounded $\PRC$ field then it has $\TP_{2}$.
Indeed, since $K$ is $\PRC$ then $L=K\left(\sqrt{-1}\right)$ is
$\PAC$ (because every finite extension of a $\PRC$ field is $\PRC$
and $L$ has no real closures at all). By \cite[Remark 16.10.3(b)]{MR2102046}
$L$ is unbounded. And of course, $L$ is interpretable in $K$. But
by the result of Chatzidakis cited above $L$ has $\TP_{2}$, thus
$K$ also has $\TP_{2}$.

\subsection{More valued fields with $\NTP_{2}$}

Is there an analogue of Fact \ref{fac: preserving NTP2} in positive
characteristic? A similar result for $\NIP$ was established in \cite[Corollaire 7.6]{MR1703196}.
\begin{conjecture}
Let $\left(K,v\right)$ be a valued field of characteristic $p>0$,
Kaplansky and algebraically maximal. Then $\left(K,v\right)$ is $\NTP_{2}$
(strong) if and only if $k$ is $\NTP_{2}$ (resp. strong).
\end{conjecture}
The following is demonstrated in \cite[Proposition 5.3]{KaScWa}.
\begin{fact}
Let $\left(K,v\right)$ be an $\NIP$ valued field of characteristic
$p>0$. Then the residue field contains $\mathbb{F}_{p}^{\alg}$ (so
in particular is infinite).
\end{fact}
Hrushovski asked if the following is true:
\begin{problem}
Assume that $\left(K,v\right)$ is an $\NTP_{2}$ (Henselian) valued
field of positive characteristic. Does it follow that the residue
field is infinite?
\end{problem}
We remark that the finite number of Artin-Schreier extensions alone
is not sufficient to conclude that the residue field is infinite:
\begin{example}
\label{ex:Arno}(Due to Arno Fehm) Let $\Omega=\left(\Ff_{p}\left(\left(t\right)\right)\right)^{\sep}$,
so the restriction map $\Gal\left(\Omega/\Ff_{p}\left(\left(t\right)\right)\right)\to\Gal\left(\Ff_{p}^{\alg}/\Ff_{p}\right)$
is onto. Let $\sigma\in\Gal\left(\Ff_{p}^{\alg}/\Ff_{p}\right)$ be
the Frobenius automorphism, and let $\tau\in\Gal\left(\Omega/\Ff_{p}\left(\left(t\right)\right)\right)$
be such that $\tau\upharpoonright\Ff_{p}^{\alg}=\sigma$. Let $F$
be the fixed field of $\tau$. Then $F$ has exactly one Artin-Schreier
extension (as $\Gal\left(\Omega/F\right)$ is pro-cyclic and $F$
is a regular extension of $\Ff_{p}$). Since $\Ff_{p}\left(\left(t\right)\right)$
is an Henselian valued field, its usual valuation extends uniquely
to an Henselian valuation on $F$. Since every element of $\Ff_{p}^{\alg}\backslash\Ff_{p}$
is moved by $\sigma$, one can see that the residue field must be
$\Ff_{p}$.
\end{example}

\begin{example}
(Due to the anonymous referee) Let $\Omega$ be the generalized power
series $\Ff_{p}^{\alg}\left(\left(t^{\Qq}\right)\right)$ --- the
field of formal sums $\sum a_{i}t^{i}$ with well-ordered support
where $i\in\Qq$ and $a_{i}\in\Ff_{p}^{\alg}$. This field is algebraically
closed. Let $\tau\in\Aut\left(\Omega\right)$ be the map $\sum a_{i}t^{i}\mapsto\sum a_{i}^{p}t^{i}$.
Let $F$ be the fixed field of $\tau$, so $F=\Ff_{p}\left(\left(t^{\Qq}\right)\right)$.
Then $F$ is Henselian with residue field $\Ff_{p}$ and (as in Example
\ref{ex:Arno}) has exactly one Artin-Schreier extension. 
\end{example}

\subsection{Definable envelopes}

Assume that we are given a subgroup of an $\NTP_{2}$ group. Is it
possible to find a \emph{definable} subgroup which is close to the
subgroup we started with and satisfies similar properties?
\begin{fact}

\begin{enumerate}
\item \cite{ShelahDependentCont,Ricardo} If $G$ is a group definable in
an $\NIP$ theory and $H$ is a subgroup which is abelian (nilpotent
of class $n$; normal and soluble of derived length $n$) then there
is a definable group containing $H$ which is also abelian (resp.
nilpotent of class $n$; normal and soluble of derived length $n$).
\item \cite{milliet:hal-00657716} Let $G$ be a group definable in a simple
theory and let $H$ be a subgroup of $G$.

\begin{enumerate}
\item If $H$ is nilpotent of class $n$, then there is a definable (with
parameters from $H$) nilpotent group of class at most $2n$, finitely
many translates of which cover $H$. If $H$ is in addition normal,
then there is a definable normal nilpotent group of class at most
$3n$ containing $H$.
\item If $H$ is a soluble of class $n$, then there is a definable (with
parameters from $H$) soluble group of derived length at most $2n$,
finitely many translates of which cover $H$. If $H$ is in addition
normal, then there is a definable normal soluble group of derived
length at most $3n$ containing $H$.
\end{enumerate}
\end{enumerate}
\end{fact}
Thus it seems very natural to make the following conjecture.
\begin{conjecture}
Let $G$ be an $\NTP_{2}$ group and assume that $H$ is a subgroup.
If $H$ is nilpotent (soluble), then there is a definable nilpotent
(resp. soluble) group finitely many translates of which cover $H$.
If $H$ is in addition normal, then there is a definable normal nilpotent
(resp. soluble) group containing $H$.
\end{conjecture}

\subsection{Hrushovski's stabilizer theorem}

Let $I$ be an ideal in the Boolean algebra of definable sets in a
fixed variable $x$, with parameters from the monster model (i.e.
$\emptyset\in I$;$\phi\left(x,a\right)\vdash\psi\left(x,b\right)$
and $\psi\left(x,b\right)\in I$ imply $\phi\left(x,a\right)\in I$;$\phi\left(x,a\right)\in I$
and $\psi\left(x,b\right)\in I$ imply $\phi\left(x,a\right)\lor\psi\left(x,b\right)\in I$).
An ideal $I$ is invariant over a set $A$ if $\phi\left(x,a\right)\in I$
and $a\equiv_{A}b$ implies $\phi\left(x,b\right)\in I$. An $A$-invariant
ideal is called $\Sone$ if for every sequence $\left(a_{i}\right)_{i\in\omega}$
indiscernible over $A$, $\phi\left(x,a_{0}\right)\land\phi\left(x,a_{1}\right)\in I$
implies $\phi\left(x,a_{0}\right)\in I$. A partial type $q\left(x\right)$
over $A$ is called wide (or $I$-wide) if it implies no formula in
$I$. 

In the following, $\tilde{G}$ is a subgroup of some definable group,
generated by some definable set $X$. 
\begin{fact}
\cite[Theorem 3.5]{MR2833482} Let $M$ be a model, $\mu$ an $M$-invariant
$\Sone$ ideal on definable subsets of $\widetilde{G}$, invariant
under (left or right) translations by elements of $\widetilde{G}$.
Let $q$ be a wide type over $M$ (contained in $\widetilde{G}$).
Assume:
\begin{lyxlist}{00.00.0000}
\item [{$(F)$}] There exist two realizations $a,b$ of $q$ such that
$\tp\left(b/Ma\right)$ does not fork over $M$ and $\tp\left(a/Mb\right)$
does not fork over $M$.
\end{lyxlist}
Then there is a wide type-definable over $M$ subgroup $S$ of $G$.
We have $S=\left(q^{-1}q\right)^{2}$; the set $qq^{-1}q$ is a coset
of $S$. Moreover, $S$ is normal in $\widetilde{G}$, and $S\setminus q^{-1}q$
is contained in a union of non-wide $M$-definable sets.
\end{fact}
In \cite{CheKap} it is proved that if $M$ is a model of an $\NTP_{2}$
theory and $q\in S\left(M\right)$, then it has a global strictly
invariant extension $p\in S\left(\C\right)$ (meaning that $p$ is
an $M$-invariant type and for every $N\supseteq M$ and $a\models p|_{N}$
we have $\tp\left(N/Ma\right)$ does not fork over $M$). It thus
follows that the assumption $(F)$ is always satisfied in $\NTP_{2}$
theories. In \cite[Section 2 + discussion before Proposition 3.5]{CheBY}
it is proved that in an $\NTP_{2}$ theory, the ideal of formulas
forking over a model $M$ is $S1$. However, in general the ideal
of forking formulas is not invariant under the action of the definable
group. By \cite[Theorem 3.5, Remark (4)]{MR2833482} the assumption
of invariance under the action of $\widetilde{G}$ can be replaced
by the existence of an $f$-generic extension of $q$. It seems interesting
to find a right version of this result generalizing the theory of
stabilizers in simple theories \cite{MR1649061}.

\bibliographystyle{alpha}
\bibliography{common}

\end{document}